\documentclass[twoside]{amsart}

\numberwithin{equation}{section}
\usepackage{amsmath,amssymb,amsfonts,amsthm,latexsym}
\usepackage{amscd,graphicx,color,enumerate}
\usepackage[all]{xy}

\newtheorem{theorem}{Theorem}[section]
\newtheorem{lemma}[theorem]{Lemma}
\newtheorem{proposition}[theorem]{Proposition}
\newtheorem{corollary}[theorem]{Corollary}

\theoremstyle{definition}

\theoremstyle{remark}

\numberwithin{equation}{section}

\newcommand{\C}{\mathbb{C}}

\newcommand{\R}{\mathbb{R}}

\newcommand{\Z}{\mathbb{Z}}

\newcommand{\U}{\operatorname{U}}
\newcommand{\SU}{\operatorname{SU}}

\newcommand{\SO}{\operatorname{SO}}
\newcommand{\OO}{\operatorname{O}}

\newcommand{\Span}{\operatorname{Span}}

\newcommand{\co}{\colon\thinspace}
\newcommand{\git}{/\!\!/}

\newcommand{\Hilb}{\operatorname{Hilb}}

\newcommand{\be}{\boldsymbol{e}}
\newcommand{\bp}{\boldsymbol{p}}
\newcommand{\bq}{\boldsymbol{q}}
\newcommand{\bx}{\boldsymbol{x}}
\newcommand{\by}{\boldsymbol{y}}
\newcommand{\bY}{\boldsymbol{Y}}
\newcommand{\bz}{\boldsymbol{z}}
\newcommand{\bw}{\boldsymbol{w}}
\newcommand{\bJ}{\boldsymbol{J}}
\newcommand{\bQ}{\boldsymbol{Q}}

\begin{document}

\title{Symplectic reduction at zero angular momentum}

\author{Joshua Cape}
\address{Department of Applied Mathematics and Statistics,
Johns Hopkins University, 3400 N. Charles St, Baltimore, MD 21218, USA}
\email{joshua.cape@jhu.edu}

\author{Hans-Christian Herbig}
\address{Departamento de Matem\'{a}tica Aplicada,
Av. Athos da Silveira Ramos 149,
Centro de Tecnologia - Bloco C,
CEP: 21941-909 - Rio de Janeiro, Brazil}
\email{herbig@imf.au.dk}

\author{Christopher Seaton}
\address{Department of Mathematics and Computer Science,
Rhodes College, 2000 N. Parkway, Memphis, TN 38112, USA}
\email{seatonc@rhodes.edu}

\keywords{symplectic reduction, moment map, angular momentum, $O_n$-representation, rational singularities}
\subjclass[2010]{Primary 53D20, 13A50; Secondary 57S15, 37J15, 20G20}
\thanks{HCH has been supported by the grant GA CR P201/12/G028 and by the grants
FAPESP 2014/20191-8 and 2014/00250-0.
CS was supported by the E.C. Ellett Professorship in Mathematics.}

\begin{abstract}
We study the symplectic reduction of the phase space describing $k$ particles
in $\R^n$ with total angular momentum zero.  This corresponds to the singular
symplectic quotient associated to the diagonal action of $\OO_n$ on $k$
copies of $\C^n$ at the zero value of the homogeneous quadratic moment map.
We give a description of the ideal of relations of the ring
of regular functions of the symplectic quotient. Using this description, we
demonstrate $\Z^+$-graded regular symplectomorphisms among the $\OO_n$- and
$\SO_n$-symplectic quotients and determine which of these quotients are graded
regularly symplectomorphic to linear symplectic orbifolds. We demonstrate that
when $n \leq k$, the zero fibre of the moment map has rational singularities
and hence is normal and Cohen-Macaulay. We also demonstrate that
for small values of $k$, the ring of regular functions on the symplectic
quotient is graded Gorenstein.
\end{abstract}

\maketitle


\section{Introduction}
\label{sec:Intro}

In this note, we examine geometric and algebraic properties of symplectic quotients
corresponding to $k$ particles moving in $\R^n$ with zero total angular momentum.  Previous studies of
these spaces by M. Gotay et. al. \cite{BosGotayReduc, GotayHomYM, ArmsGotayJennings}
and by J. Huebschmann, see e.g. \cite{HuebschSingPoissonCertainRep}, were motivated
by gauge theory. More specifically \cite[Theorem 4]{HuebschSingPoissonCertainRep},
the cases $n=2,3$ occur as local models for the strata of the moduli space of flat
$\operatorname{SU}_2$-connections on a closed hyperelliptic Riemann surface of genus $g$,
the number of particles $k$ being linked to the genus by $k=g$ or $k=g-1$, respectively.
Moreover, as an illustration to the seminal paper \cite{SjamaarLerman} of E. Lerman and
R. Sjamaar, the symplectic quotients at zero angular momentum have been studied in
\cite{LermanMontgomerySjamaar}, see also the survey \cite{HuebschMemoirs}.

Here, we expand on the discussion in \cite{LermanMontgomerySjamaar} of the symplectomorphism problem
(see Section \ref{sec:Q}) for symplectic quotients at zero angular momentum and explain in more detail
the distinction between reduction by $\operatorname{O}_n$ and by $\operatorname{SO}_n$ (see Section
\ref{sec:OrbitTypes}). Our strategy for constructing maps between symplectic quotients is based on
polynomial invariant theory and emphasizes the role of the (graded) $\R$-algebra of regular functions
$\R[M_0]$ on the symplectic quotient $M_0$. This approach has been already advocated in
\cite{FarHerSea,HerbigSeatonHSeries,HerbigSchwarzSeaton,HerbigSeatonImpos}; the requisite material will
be recalled in Section \ref{sec:Setup}. Essential for the construction of the symplectomorphisms will be
a system of real polynomials $Q_{i,j}$ in the invariants whose locus coincides with that of the moment map.

Moreover, we show in Section \ref{sec:OrbitTypes} that the symplectic quotient is symplectomorphic to an
orbifold precisely when $k=1$ or $n=1$, illustrating that in this case,
the non-orbifold conditions of \cite{HerbigSchwarzSeaton} are necessary but not sufficient;
this is also illustrated in that reference with the case of $\SU_2$-symplectic quotients.
In Section \ref{sec:RatSing}, we show that when $n \leq k$, the zero fibre of the moment map has
\emph{rational singularities}, which in particular implies that the corresponding ideal is
integrally closed and Cohen-Macaulay. Employing Boutot's Theorem \cite{Boutot}
as well as the aforementioned symplectomorphisms, it follows that the symplectic quotients
by $\OO_n$ and $\SO_n$ have rational singularities for each $n$ and $k$, a result which also
follows from the results of \cite{TerpereauThesis} and \cite{BeauvilleSympSing} for the
corresponding complex symplectic quotients. Note that in \cite{SjamaarHoloSlice}, symplectic
quotients are shown to have rational singularities in terms of the
complex analytic structure they inherit from the corresponding GIT quotient via the Kempf-Ness theorem;
see Section \ref{sec:Setup}. The results discussed here, on the other hand, are in terms of the real
structure of the symplectic quotient or the complex analytic structure of the complex symplectic quotient.

In Section \ref{sec:Computations}, the experimental part of the paper, we address the open question
of determining the relations of the intersection of the ring of invariants with the vanishing ideal of the
zero fibre of the moment map, i.e. the real radical of the ideal described by Theorem \ref{thrm:IdealRelations}.
In those cases where we can calculate the generators with the computer
algebra software \emph{Macaulay2} \cite{M2} and \emph{Mathematica} \cite{Mathematica}
(namely, $k\leq 3$ and some cases when $k=4$) we show that there are no additional relations, and
$\R[M_0]$ is actually graded Gorenstein.
We discuss the influence of the choice of a term order on the efficiency of the respective Gr\"{o}bner
basis calculations. That $\R[M_0]$ is graded Gorenstein with rational singularities is expected in
a much wider generality.


\section*{Acknowledgements}
We would like to thank Anurag Singh and Srikanth Iyengar for advice concerning rational
singularities, Maria Aparecida Soares Ruas for guiding us to the right literature on
determinantal varieties, and Gerald Schwarz, from whom we learned the
techniques used in the proof of Theorem \ref{thrm:OrbifoldCriteria}.
HCH has profited from reading an unpublished manuscript of Claudio
Emmrich on the case $k=2$, $n=3$. CS would like to thank the Instituto de Ci\^{e}ncias
Matem\'{a}tica e de Computa\c{c}\~{a}o, Universidade de S\~{a}o Paulo for their hospitality
during the completion of this project. This paper developed from the JC's
senior seminar project in the Rhodes College Department of Mathematics and Computer
Science, and the authors gratefully acknowledge the support of the department and college
for these activities.


\section{Setup and notation}
\label{sec:Setup}

Let $k$ be a positive integer. Throughout this paper, we consider the phase space of
$k$ particles in $\R^n$. We let $V_{k,n} = \R^{2kn}$ denote the phase space
of configurations of the $k$ particles; throughout, we will suppress the indices
$k,n$ when it will cause no confusion. For each $\ell = 1,\ldots,k$, we let
$\bq_\ell = (q_{\ell,1},\ldots,q_{\ell,n}) \in \R^n$ denote the position and
$\bp_\ell = (p_{\ell,1},\ldots,p_{\ell,n}) \in \R^n$ denote the momentum
of the $\ell$th particle; we will use Greek letters to indicate the coordinates in
$\R^n$ and Roman letters for the particle numbering. Then $V$ carries the symplectic
structure associated to it as the underlying real space of the complex vector space
$\C^{kn}$ with coordinates $q_{\ell,\alpha} + \sqrt{-1}p_{\ell,\alpha}$ and the
standard Hermitian structure. We use $(\bq,\bp)$
to abbreviate the coordinates $(\bq_1,\bp_1,\ldots,\bq_k,\bp_k)$ for $V$. For certain
computations, it will be convenient to define $\by_{2\ell-1} = \bq_\ell$ and
$\by_{2\ell} = \bp_\ell$ so that we may use coordinates $\by = (\by_1,\ldots,\by_{2k})$
where each $\by_i\in\R^n$ as well, with $\by_i = (y_{i,1},\ldots,y_{i,n})$.
We will throughout consider the diagonal action of the
real orthogonal group $\OO_n :=\OO_n(\R)$ on $V$ given by the standard $\OO_n$-action on
each $\R^n$-factor. Representing a point $(\bq,\bp)\in V$ as an $n\times 2k$ matrix,
this action corresponds to multiplication on the left by elements of $\OO_n$. For the
constant Poisson bracket $\{\,,\,\}$ on $\R[V_{k,n}]$ we use the convention
$\{q_{\ell_1,\alpha},p_{\ell _2,\beta}\}:=\delta_{\ell_1,\ell_2}\delta_{\alpha,\beta}$
for $i,j=1,\dots,k$ and $\alpha,\beta=1,\dots,n$; all other brackets between linear
coordinates are supposed to vanish.

The action of $\OO_n$ (respectively $\SO_n$) on the real vector space $V$ extends
to an action of the complexification $\OO_n(\C)$ (respectively $\SO_n(\C)$) on
the complex vector space $V\otimes_\R\C$, and hence the real invariant polynomials
$\R[V]^{\OO_n}$ can be computed in terms of the complex invariants
$\C[V\otimes_\R\C]^{\OO_n(\C)}$; see \cite[Proposition 5.8(1)]{GWSliftingHomotopies}.
By the First and Second Fundamental Theorems of Invariant Theory for $\OO_n$, see
\cite[Sections 9.3 and 9.4]{PopovVinberg} and \cite[Sections 9 and 17]{Weyl},
the invariants $\R[V]^{\OO_n}$ are generated by the scalar products
$x_{i,j} := \langle\by_i,\by_j\rangle$, $1\leq i \leq j \leq 2k$. The relations among these
invariants, which we refer to as \emph{off-shell relations}, are generated by the
$(n+1)\times(n+1)$-minors of the symmetric $2k\times 2k$ matrix $X:=(x_{i,j})$; see the
definition of the \emph{shell} below for an explanation of this language. Hence
$\R[V]^{\OO_n}$ is a polynomial algebra if and only if $n \geq 2k$. In the same way,
$\R[V]^{\SO_n}$ is generated by the $x_{i,j}$ as well as determinants
$\det(\by_{i_1},\ldots,\by_{i_n})$ with $1\leq i_1 < \cdots < i_n \leq 2k$.
The invariant rings  $\R[V]^{\OO_n}\subset \R[V]^{\SO_n}$ are by construction
Poisson subalgebras of $\R[V]$. This can also be expressed in terms of commutation
relations among the invariants, e.g. for $\ell_1,\ell_2,\ell_3,\ell_4=1,2,\dots,k$:
\begin{equation}
\label{eq:PoissonBrackets}
\begin{split}
    \{x_{2\ell_1-1,2\ell_2-1},x_{2\ell_3,2\ell_4}\}
    =&  \delta_{\ell_1,\ell_3}x_{2\ell_2-1,2\ell_4} + \delta_{\ell_1,\ell_4}x_{2\ell_2-1,2\ell_3}
        \\
        &+ \delta_{\ell_2,\ell_3}x_{2\ell_1-1,2\ell_4}+\delta_{\ell_2,\ell_4}x_{2\ell_1-1,2\ell_3},
        \\
    \{x_{2\ell_1-1,2\ell_2},x_{2\ell_3-1,2\ell_4}\}
    =&  \delta_{\ell_1,\ell_4}x_{2\ell_3-1,2\ell_2} - \delta_{\ell_2,\ell_3}x_{2\ell_1-1,2\ell_4},
        \\
    \{x_{2\ell_1-1,2\ell_2},x_{2\ell_3-1,2\ell_4-1}\}
    =&  -\delta_{\ell_1,\ell_4}x_{2\ell_1-1,2\ell_3-1} - \delta_{\ell_2,\ell_3}x_{2\ell_1-1,2\ell_4-1},
        \\
    \{x_{2\ell_1-1,2\ell_2},x_{2\ell_3,2\ell_4}\}
    =&  \delta_{\ell_1,\ell_4}x_{2\ell_1-1,2\ell_3}+\delta_{\ell_2,\ell_3}x_{2\ell_1-1,2\ell_4}.
\end{split}
\end{equation}

Identifying the Lie algebra $\mathfrak{o}_n^\ast$ with $\bigwedge^2 \R^n$ in the standard
way, the moment map $\bJ=\bJ_{k,n}\co V \to \mathfrak{o}_n^\ast$ can be expressed as
\[
    \bJ(\bq,\bp)
    =
    \sum\limits_{\ell=1}^k \bq_\ell\wedge\bp_\ell.
\]
For $1\leq \alpha < \beta \leq n$, we let
\begin{equation}
\label{eq:DefJ}
\begin{split}
    J_{\alpha,\beta}(\bq, \bp)
    =&
    J_{k,n,\alpha,\beta}(\bq, \bp)
    =
    \sum\limits_{\ell=1}^k q_{\ell,\alpha} p_{\ell,\beta} - q_{\ell,\beta} p_{\ell,\alpha}
    \\\nonumber =&
    \sum\limits_{\ell=1}^k y_{2\ell-1,\alpha} y_{2\ell,\beta} - y_{2\ell-1,\beta} y_{2\ell,\alpha}
\end{split}
\end{equation}
denote the $\be_\alpha\wedge\be_\beta$-component of
$\bJ(\bq, \bp)$ where the $\be_\alpha$ denote the standard
basis vectors of $\R^n$. We let
$\mathcal{J} = \mathcal{J}_{k,n} = \langle J_{k,n} \mid 1\leq i < j \leq 2k \rangle$
denote the ideal of $\R[V]$ generated by the components of the moment map.  For all
$\alpha,\beta,\gamma,\epsilon=1,2,\dots,n$ the moment map fulfills the commutation relations
\[
    \{J_{\alpha,\beta},J_{\gamma,\epsilon}\}
    =   \delta_{\alpha,\epsilon}J_{\beta,\gamma} + \delta_{\beta,\gamma}J_{\epsilon,\alpha}
        + \delta_{\alpha,\gamma}J_{\beta,\epsilon}+\delta_{\beta,\epsilon}J_{\alpha,\gamma},
\]
which implies that $\{\mathcal{J},\mathcal{J}\}\subset \mathcal{J}$.

Let $Z_{k,n} = Z = \bJ^{-1}(0)$ denote the zero fiber of the moment map, which
we refer to as the \emph{shell}. The phase space of configurations with angular
momentum zero is the \emph{symplectic quotient}
\[
    M_0 = M_{0,k,n} :=  Z/\OO_n,
\]
and is our primary object of study. We regard it as a \emph{Poisson differential
space} with a \emph{global chart} defined in terms of the invariants $x_{i,j}$;
we briefly recall the basics of such spaces and refer the reader to
\cite[Section 4]{FarHerSea} or \cite[Section 2.1]{HerbigSeatonImpos} for more details.
The \emph{Hilbert embedding} $V/\OO_n\to \R^{2k\choose 2}$ with coordinates $x_{i,j}$
restricts to an embedding of $M_0$ into $\R^{2k\choose 2}$, realizing
$M_0$ as a semialgebraic set. We let $\mathcal{I}_Z$ denote the vanishing ideal of
$Z$ in $\R[V]$ and $\mathcal{I}_Z^{\OO_n} := \mathcal{I}_Z\cap\R[V]^{\OO_n}$. Then
the \emph{Poisson algebra of real regular functions on $M_0$} is
\[
    \R[M_0] :=  \R[V]^{\OO_n}/ \mathcal{I}_Z^{\OO_n}.
\]
The algebra $\R[M_0]$ is a Poisson subalgebra of the algebra $\mathcal{C}^\infty(M_0)$
which is defined as the quotient of $\mathcal{C}^\infty(V)^{\OO_n}$ by the ideal of
smooth invariant functions that vanish on $Z$; see \cite{SjamaarLerman}. We refer
to the relations among the $x_{i,j}$ in $\R[M_0]$, which include the off-shell relations
as well as the elements of $\mathcal{I}_Z^{\OO_n}$, as \emph{on-shell relations}.

Given two such Poisson differential spaces with global charts, an
isomorphism between the algebras of regular functions induces a homeomorphism
between the Zariski closures of the associated semialgebraic sets. Letting
$\Z^+$ denote the nonnegative integers, a ($\Z^+$-graded)
isomorphism between the algebras of regular functions is a
\emph{($\Z^+$-graded) regular diffeomorphism} if the corresponding homeomorphism
restricts to a homeomorphism between the semialgebraic sets themselves, and a
\emph{($\Z^+$-graded) regular symplectomorphism} if the isomorphism on algebras of
regular functions is Poisson. By \cite[Theorem 6]{FarHerSea}, a $\Z^+$-graded regular
symplectomorphism induces an isomorphism of Poisson differential spaces and in
particular a Poisson isomorphism between the algebras of smooth functions.

The infinitesimal actions, and hence moment maps, of the $\OO_n$- and
$\SO_n$-actions on $V$ coincide. Hence the symplectic quotient of the Hamiltonian
$\SO_n$-space $V$ is given by
\[
    M_0^{\SO} = M_{0,k,n}^{\SO} :=  Z/\SO_n
\]
with Poisson algebra of real regular functions
\[
    \R[M_0^{\SO}] := \R[V]^{\SO_n}/\mathcal{I}_Z^{\SO_n}.
\]

In addition, we may consider $V$ as the underlying real space of the complex
vector space $\C^{kn}$ with coordinates $q_{\ell,\alpha} + \sqrt{-1}p_{\ell,\alpha}$ as above.
Then $Z$ is the \emph{Kempf-Ness set} associated to the action of $\OO_n(\C)$
(or $\SO_n(\C)$) on $\C^{kn}$. It follows that the symplectic quotient $M_{0,k,n}$
(respectively $M_{0,k,n}^{\SO}$) is homeomorphic via the Kempf-Ness homeomorphism
to the affine GIT quotient $\C^{kn}\git\OO_n(\C) = \operatorname{Spec}(\C[\C^{kn}]^{\OO_n(\C)})$
(respectively $\C^{kn}\git\SO_n(\C) = \operatorname{Spec}(\C[\C^{kn}]^{\SO_n(\C)})$),
see \cite{KempfNess,GWSkempfNess}. Here, the action is again the diagonal of the standard
action on the $k$ $\C^n$-factors. The $\OO_n(\C)$-invariants are generated by the Euclidean inner
products of the vectors in $\C^n$, and the $\SO_n(\C)$-invariants are given by these and the
determinants of $n$ distinct vectors in $\C^n$ by the Fundamental Theorem of Invariant Theory
described above.

We recall the following and refer the reader to \cite{GWSlifting,HerbigSchwarz} for more details.
Let $G$ be a reductive group and $Y$ an affine $G$-variety over $\C$. Each point
in the GIT quotient $Y\git G$ corresponds to a closed $G$-orbit in $Y$, and $Y\git G$ is
stratified by the orbit types of the closed orbits. There is a unique open orbit type
in $Y\git G$, the \emph{principal orbit type}, and elements of the corresponding conjugacy class
of isotropy groups are the \emph{principal isotropy groups}. We say $Y$ has \emph{finite principal
isotropy groups (FPIG)} if the principal isotropy groups are finite, and that $Y$ is
\emph{$m$-principal} if it has FPIG and the set of principal orbits has (complex) codimension at least $m$.
Letting $Y_{(j)}$ denote the set of points in $Y$ with isotropy group of complex dimension $j$
and $c_j:= \dim_{\C} V_{(j)} - \dim_{\C} G + j$, we say $Y$ is \emph{$m$-modular} if $V_{(0)}\neq\emptyset$
and $c_j + m \leq c_0$. Finally, $Y$ is \emph{$m$-large} if it is $m$-principal and $m$-modular.

In the case that $G = \OO_n$ or $\SO_n$ and $Y = k\C^{n}$ as above,
\cite[Theorem 11.18]{GWSlifting} and \cite[3.5]{HerbigSchwarz} give criteria for $Y$ to be
$1$- or $2$-large. Moreover, \cite[Corollary 4.3]{HerbigSchwarz} implies that in the $1$-large
cases, the vanishing ideal of the (real) shell $Z_{k,n}$ in $\R[V]$ is generated by the components
of the moment map. Similarly, letting $\bJ^\C$ denote the complexification of the moment map,
\cite[Theorem 2.2]{HerbigSchwarz} describes consequences for the scheme associated to the ideal
$\mathcal{J}^\C$ generated by the components of $\bJ^\C$. We summarize these results with the following.

\begin{theorem}[\cite{GWSlifting,HerbigSchwarz}]
\label{thrm:LargeCases}
Let $\C^{kn} = k\C^n$ be equipped with the diagonal of the standard action of $\OO_n(\C)$.
\begin{itemize}
\item[({\it i.})]
        If $n\leq k+1$, then $\C^{kn}$ is $1$-large as an $\OO_n(\C)$- and $\SO_n(\C)$-module.
        Hence the ideal $\mathcal{J}^\C$ of $\C[V\otimes_{\R}\C]$ is a reduced, irreducible complete
        intersection, and the ideal $\mathcal{J}$ of $\R[V]$ is real radical.
\item[({\it ii.})]
        If $n \leq k$, then $\C^{kn}$ is $2$-large as an $\SO_n(\C)$-module.
        Hence $\C[\C^{kn}]/\mathcal{J}^\C$ is integrally closed, i.e. the corresponding
        variety is normal.
\item[({\it iii.})]
        If $n \leq k-1$, then $\C^{kn}$ is $2$-large as an $\OO_n(\C)$-module.
\end{itemize}
\end{theorem}

In particular, when $k = n$, the $\OO_n(\C)$-representation $\C^{kn}$
is not $2$-large, yet the $\SO_n(\C)$-representation is. Because
$\mathcal{J}^\C$ coincides for these two representations, the corresponding
variety is normal, illustrating that the converse of \cite[Theorem 2.2(4)]{HerbigSchwarz}
is false.


\section{The quadratic on-shell relations}
\label{sec:Q}

In this section, we consider the ideal of relations defining $\R[M_0]$ as an affine algebra.
Our goal is to express these relations in terms of
the invariants $x_{i,j}$, $1\leq i \leq j \leq 2k$.

Fix positive integers $n$ and $k$. For $1\leq i < j \leq 2k$, define
\begin{equation}
\label{eq:QDef}
    Q_{k,n,i,j} = \sum\limits_{\ell=1}^k \det
        \begin{pmatrix} x_{i, 2\ell-1} & x_{i, 2\ell} \\
                        x_{j, 2\ell-1} & x_{j, 2\ell} \end{pmatrix}
        = \sum\limits_{\ell=1}^k \det
            \begin{pmatrix} \langle\by_i,\by_{2\ell-1}\rangle & \langle\by_i,\by_{2\ell}\rangle \\
            \langle\by_j,\by_{2\ell-1}\rangle & \langle\by_j,\by_{2\ell}\rangle \end{pmatrix}.
\end{equation}

We consider the $Q_{k,n,i,j}$ as quadratic polynomials in the invariants $x_{i,j}$ or equivalently
as quartic polynomials in the $y_{i,j}$. Note that when $n = 1$, the $Q_{k,n,i,j}$ are forced to
vanish as the $2\times 2$-minors of a Gram matrix must vanish. We will first prove the following.

\begin{proposition}
\label{prop:JiffQ}
Let $k \geq 1$ and $n \geq 1$.  For a point $\by \in V$, the following
are equivalent:
\begin{itemize}
\item[({\it i.})]   $Q_{k,n,2\ell-1,2\ell}(\by) = 0$ for each $1\leq \ell \leq k$,
\item[({\it ii.})]  $Q_{k,n,i,j}(\by) = 0$ for each $1 \leq i < j \leq 2k$, and
\item[({\it iii.})] $\bJ(\by) = 0$, i.e. $J_{k,n,\alpha,\beta}(\by) = 0$
                    for each $1\leq \alpha < \beta \leq n$.
\end{itemize}
\end{proposition}

We begin with the following.

\begin{lemma}
For each $k, n\geq 1$,
\label{lem:sumQtoJ^2}
\begin{equation}
\label{eq:sumQtoJ^2}
    \sum\limits_{\ell=1}^k Q_{k,n,2\ell-1,2\ell}
    =
    \sum\limits_{1\leq \alpha < \beta \leq n}
    J_{k,n,\alpha,\beta}^2.
\end{equation}
\end{lemma}
Note that the right-hand side of Equation \eqref{eq:sumDtoJ^2} is the norm squared of the moment map,
which plays a role in the study of convexity properties of $\bJ$ \cite{KirwanConvexIII}.
\begin{proof}
When $n = 1$, the left side of Equation \eqref{eq:sumDtoJ^2} vanishes as noted above, and the right
side is the empty sum, so assume $n > 1$. We compute
\[
    \sum_{\ell=1}^k Q_{k,n,2\ell-1,2\ell}
    =
    \sum\limits_{\ell_1=1}^k \sum\limits_{\ell_2=1}^k
                \langle\bq_{\ell_1}, \bq_{\ell_2}\rangle
                \langle\bp_{\ell_1}, \bp_{\ell_2}\rangle
                -
                \langle\bq_{\ell_1}, \bp_{\ell_2}\rangle
                \langle\bp_{\ell_1}, \bq_{\ell_2}\rangle,
\]
which, after expanding and canceling, yields
\[
\begin{split}
    =&
    \sum\limits_{1 \leq\alpha < \beta\leq n} \quad
    \sum\limits_{\ell=1}^k \big( q_{\ell,\alpha}p_{\ell,\beta} - q_{\ell,\beta}p_{\ell,\alpha} \big)
    \sum\limits_{\ell=1}^k \big( q_{\ell,\alpha}p_{\ell,\beta} - q_{\ell,\beta}p_{\ell,\alpha} \big)
    \\=&
    \sum\limits_{1 \leq\alpha < \beta\leq n} \quad J_{k,n,\alpha,\beta}^2.
    \qedhere
\end{split}
\]
\end{proof}

Now, for $k \geq 1$ and $n \geq 2$, we define for each $1\leq i<j \leq 2k$ the difference
\[
    D_{k,n,i,j} = Q_{k,n,i,j} - Q_{k,n-1,i,j}.
\]
Note that $D_{k,2,i,j} = Q_{k,2,i,j}$ due to the vanishing of $Q_{k,1,i,j}$.

\begin{lemma}
\label{lem:DtoSumJ}
For $n\geq 2$, $k \geq 1$, and $1 \leq i < j \leq 2k$, we have
\begin{equation}
\label{eq:DtoSumJ}
    D_{k,n,i,j}
    =
    \sum\limits_{\alpha=1}^{n-1} (y_{i,\alpha}y_{j,n}-y_{i,n}y_{j,\alpha}) J_{k,n,\alpha,n}.
\end{equation}
\end{lemma}
\begin{proof}
For fixed $i$ and $j$, we express $D_{k,n,i,j}$ as
\[
\begin{split}
    \sum\limits_{\ell=1}^k & \left(
        \langle \by_i, \by_{2\ell-1} \rangle \langle \by_j, \by_{2\ell} \rangle
            - \left(\sum\limits_{\alpha=1}^{n-1} y_{i,\alpha}y_{2\ell-1,\alpha}\right)
            \left(\sum\limits_{\beta=1}^{n-1} y_{j,\beta}y_{2\ell,\beta}\right)\right.
    \\
        & - \left. \langle \by_i, \by_{2\ell} \rangle \langle \by_j, \by_{2\ell-1} \rangle
        - \left(\sum\limits_{\alpha=1}^{n-1} y_{i,\alpha}y_{2\ell,\alpha}\right)
            \left(\sum\limits_{\beta=1}^{n-1} y_{j,\beta}y_{2\ell-1,\beta}\right)\right).
\end{split}
\]
Expanding, the terms where both $\R^n$-indices are $n$ cancel, and we continue
\[
\begin{split}
    =&
    \sum\limits_{\ell=1}^k \sum\limits_{\alpha=1}^{n-1} \big(
        y_{i,n}y_{2\ell-1,n} y_{j,\alpha}y_{2\ell,\alpha} + y_{i,\alpha}y_{2\ell-1,\alpha} y_{j,n}y_{2\ell,n}
        \\&\quad\quad\quad\quad\quad\quad
        - y_{i,n}y_{2\ell,n} y_{j,\alpha}y_{2\ell-1,\alpha}
        - y_{i,\alpha}y_{2\ell,\alpha} y_{j,n}y_{2\ell-1,n}\big)
    \\=&
    \sum\limits_{\ell=1}^k \sum\limits_{\alpha=1}^{n-1} (y_{i,\alpha}y_{j,n}-y_{i,n}y_{j,\alpha})
        (y_{2\ell-1,\alpha}y_{2\ell,n} - y_{2\ell-1,n}y_{2\ell,\alpha})
    \\=&
    \sum\limits_{\alpha=1}^{n-1} (y_{i,\alpha}y_{j,n}-y_{i,n}y_{j,\alpha}) J_{k,n,\alpha,n}.
    \qedhere
\end{split}
\]
\end{proof}

In particular, the following is an immediate consequence of Equation \eqref{eq:DtoSumJ}.

\begin{corollary}
\label{cor:sumDtoJ^2}
For $n\geq 2$, $k \geq 1$, we have
\begin{equation}
\label{eq:sumDtoJ^2}
    \sum\limits_{\ell=1}^k D_{k,n,2\ell-1,2\ell}
    =
    \sum\limits_{\alpha=1}^{n-1} J_{k,n,\alpha,n}^2.
\end{equation}
\end{corollary}

\begin{proof}[Proof of Proposition \ref{prop:JiffQ}]
Fix $k \geq 1$.  If $n = 1$, then as all $2\times 2$-minors of a Gram matrix of vectors in
$\R^1$ must vanish and $\bJ$ is the zero map, conditions ({\it i.}), ({\it ii.}), and ({\it iii.})
hold trivially for each point in $V$.  Hence, we may assume $n \geq 2$.

The fact that ({\it iii.})$\Rightarrow$({\it ii.}) can be seen by induction on $n$ and
Lemma \ref{lem:DtoSumJ}.  Specifically, the base case $n = 2$ follows from the fact
that if $\bJ(\by) = 0$, then the right side of Equation \eqref{eq:sumDtoJ^2}
vanishes, yielding $Q_{k,2,i,j}(\by) = D_{k,2,i,j}(\by) = 0$ for each $i,j$.
Assuming each $Q_{k,n-1,i,j}(\by) = 0$ for some $n$, Equation \eqref{eq:DtoSumJ}
then implies that each $Q_{k,n,i,j}(\by) = 0$.  The implication
({\it ii.}) $\Rightarrow$ ({\it i.}) is obvious.  Then ({\it i.})$\Rightarrow$({\it iii.})
is an immediate consequence of Lemma \ref{lem:sumQtoJ^2}.
\end{proof}

It follows that the $Q_{i,j}$ are elements of the ideal $\mathcal{I}_Z^{\OO_n}$,
and moreover that $\mathcal{I}_Z^{\OO_n}$ is the real radical of the ideal generated by
the $Q_{i,j}$. In addition, note that when $n < 2k$, the off-shell relations of the $x_{i,j}$
in the quotient $\R[V]^{\OO_n}$ are generated by the $(n+1)\times(n+1)$-minors of $X$,
see Section \ref{sec:Setup}, so these minors must be on-shell relations as
well. When $k \leq n$, a stronger statement is true. By an application of
Cartan's Lemma, if $\bJ(\by) = 0$, then the subspace of $\R^n$ spanned by
$\{\by_1,\ldots,\by_{2k}\}$ has dimension at most $k$; see
\cite[page 23]{LermanMontgomerySjamaar}. Therefore, it must be that the $(k+1)\times(k+1)$-minors
of $X$ vanish in $\R[M_0]$, and hence are contained in the real radical of the ideal
generated by the $Q_{i,j}$.

Our next goal is to demonstrate that in general, the $(k+1) \times (k+1)$-minors of $X$
are in fact in the ideal generated by the $Q_{i,j}$ themselves, and not simply elements
of its real radical. We state and prove the following proposition more generally, as it
holds for an arbitrary matrix that need not be symmetric nor square. Specifically, let
$m$ be a positive integer, and let $x_{i,j}$, $1\leq i \leq m$ and $1\leq j \leq 2k$
denote arbitrary variables. Then we may consider the ideal of
$\R[x_{i,j}\mid 1\leq i \leq m, 1 \leq j \leq 2k]$ generated by the $Q_{i,j}$
for $1\leq i \leq m$ and $1\leq j \leq 2k$ as defined by Equation \eqref{eq:QDef}.
Note that, when $m = 2k$ and the $x_{i,j}$ are the $\OO_n$-invariants as above, this
result is trivial unless $k \leq n$. Specifically, if $n < k$, then the $(n+1)\times(n+1)$-minors,
which are off-shell relations among the $x_{i,j}$, generate the $(k+1)\times(k+1)$-minors.

\begin{proposition}
\label{prop:QsGenMinors}
Let $k, m \geq 1$. The ideal of the polynomial ring $\R[x_{i,j}\mid 1\leq i \leq m, 1\leq j \leq 2k]$
generated by the $Q_{i,j}$ contains the $(k+1)\times(k+1)$-minors of the matrix $X$.
\end{proposition}

To prove Proposition \ref{prop:QsGenMinors}, consider an arbitrary
$m\times 2k$ matrix $X$, and let $R = \R[x_{i,j} | 1\leq i \leq m, 1\leq j \leq 2k]$ denote
corresponding polynomial ring. We work in $R^m$, which we consider both as an $R$-module and
an $\R$-vector space, and consider each column of $X$ as an element of $R^m$. In particular, we let
$P_{m,k}$ denote the sub-$\R$-vector space of $R^m$ given by the $\R$-linear span of the $2k$ columns
of $X$.

Let $\bx_i \in R^m$ denote the $i$th column of $X$, and let $\be_i \in R^m$ denote the element
containing $1$ in the $i$th position and $0$ elsewhere as usual. Define
\[
    \bQ_k = \sum\limits_{\ell=1}^k \bx_{2\ell-1}\wedge\bx_{2\ell}.
\]
Note that in the case that $m = 2k$ and $x_{i,j} = \langle\by_i,\by_j\rangle$ as above, we have
\[
    \bQ_k = \sum\limits_{1\leq i < j \leq 2k}
        Q_{i,j} \be_i \wedge \be_j.
\]
We claim the following.

\begin{lemma}
\label{lem:Wedge}
Let $X$ be an $m\times 2k$ matrix with columns $\bx_1\,\ldots,\bx_{2k}$.
For any element $\tau \in \bigwedge^{k+1} P_{m,k}$, there is an element
$\omega \in \bigwedge^{k-1} P_{m,k}$ such that $\bQ_k\wedge \omega = \tau$.
\end{lemma}
\begin{proof}
Choose any positive integer $m$. The proof is by induction on $k$. If $k = 1$,
then $\bQ_1 = \bx_1\wedge\bx_2$, and $\bigwedge^{2} P_{m,2}$ is spanned by $\bx_1\wedge\bx_2$.
Therefore, $\tau$ must be a scalar multiple of $\bQ_1$, and setting $\omega$ to be that scalar
yields the result.

Now, fix $k \geq 2$, and assume that for any $s < k$ and any $\tau\in\bigwedge^{s+1} P_{m,s}$, there is a
$\sigma\in\bigwedge^{s-1} P_{m,s}$ such that $\bQ_{s}\wedge\sigma = \tau$. Choose a
$\tau\in\bigwedge^{k} P_{m,s}$. By linearity,
we may assume that $\tau$ is a monic $(k+1)$-blade, i.e.
$\tau = \bx_{i_1}\wedge\bx_{i_2}\wedge\cdots\wedge\bx_{i_{k+1}}$.

Note that $\bQ_k$ is invariant under permuting pairs of indices of the form $(2\ell-1,2\ell)$
for $\ell=1,\ldots,k$, so we may apply such a permutation to $\tau$ for simplicity.
Then there are two cases.

\noindent\textbf{Case 1:} {\it For at least one pair $(2\ell-1,2\ell)$, one element of the pair
appears in $\{i_1,\ldots,i_{k+1}\}$ while the other does not.}
In this case, by permuting pairs, we can assume
$i_k$ and $i_{k+1}$ are not in a pair of the form $(2\ell-1, 2\ell)$.
Let $\bx_j$ denote the other element of the pair containing $\bx_{i_{k+1}}$.
Let $\bQ_{k-1}^{\neg i_{k+1},j}$ denote $\bQ_{k}$ after removing the term
$\pm \bx_{i_{k+1}}\wedge\bx_j$. By the inductive hypothesis, after relabeling the
variables to see that $P_{m,k-1}$ is isomorphic to
$\Span_\R(\{\bx_1,\ldots\bx_{2k}\} \smallsetminus\{\bx_{i_{k+1}},\bx_j \})$,
there is a
$\sigma \in \bigwedge^{k-2} \Span_\R(\{\bx_1,\ldots\bx_{2k}\} \smallsetminus\{\bx_{i_{k+1}},\bx_j \})$
such that
\[
    \bQ_{k-1}^{\neg i_{k+1},j} \wedge\sigma
    = \bx_{i_1}\wedge\cdots\wedge\bx_{i_k}.
\]
Let $\omega = \sigma\wedge \bx_{i_{k+1}}$, and then we compute
\[
    \bQ_k \wedge \omega
    =
    (\bQ_{k-1} \pm \bx_{i_{k+1}}\wedge\bx_j) \wedge \sigma\wedge\bx_{i_{k+1}}
    =
    \tau.
\]
Hence the inductive step is satisfied in this case.

If the $i_1,\ldots,i_{k+1}$ form a union of pairs of the form $(2\ell-1,2\ell)$, then up to
permuting pairs, we may assume the following.

\noindent\textbf{Case 2:} {\it $k$ is odd, and $(i_1,\ldots,i_{k+1}) = (1,2,\ldots,k+1)$.}
In this case, we will describe $\sigma$ explicitly. Let $r = (k+1)/2$ and define
\[
    \sigma    =
    \sum\limits_{t=0}^{r-1}
    \frac{(-1)^t(r-t-1)!t!}{r!} \Big(
        \sum\limits_{\substack{A\subseteq\{1,\ldots,r\}\\ |A|=r-t-1}}
        \;
        \sum\limits_{\substack{B\subseteq\{r+1,\ldots,k\}\\ |B|=t}}
        \;
        \bigwedge\limits_{\ell\in A\cup B}\bx_{2\ell-1}\wedge\bx_{2\ell}
    \Big).
\]
To simplify notation in the following computation, $A$ will always denote subsets of $\{1,\ldots,r\}$
and $B$ will denote subsets of $\{r+1,\ldots,k\}$. We compute $\bQ_k \wedge \sigma$ to be
\[
\begin{split}
    \Big(\sum\limits_{\ell=1}^k & \bx_{2\ell-1}\wedge\bx_{2\ell}\Big)
    \wedge
        \sum\limits_{t=0}^{r-1}
        \frac{(-1)^t(r-t-1)!t!}{r!} \Big(
        \sum\limits_{\substack{|A|=r-t-1 \\ |B|=t}}
        \quad
        \bigwedge\limits_{\ell\in A\cup B}\bx_{2\ell-1}\wedge\bx_{2\ell}
    \Big)
    \\=&
    \Big(\sum\limits_{\ell=1}^r \bx_{2\ell-1}\wedge\bx_{2\ell}\Big)
    \wedge
        \sum\limits_{t=0}^{r-1}
        \frac{(-1)^t(r-t-1)!t!}{r!} \Big(
        \sum\limits_{\substack{|A|=r-t-1 \\ |B|=t}}
        \;
        \bigwedge\limits_{\ell\in A\cup B}\bx_{2\ell-1}\wedge\bx_{2\ell}
    \Big)
    \\+&
    \Big(\sum\limits_{\ell=r+1}^k \bx_{2\ell-1}\wedge\bx_{2\ell}\Big)
        \wedge
        \sum\limits_{t=0}^{r-1}
        \frac{(-1)^t(r-t-1)!t!}{r!} \Big(
        \sum\limits_{\substack{|A|=r-t-1 \\ |B|=t}}
        \;
        \bigwedge\limits_{\ell\in A\cup B}\bx_{2\ell-1}\wedge\bx_{2\ell}
    \Big).
\end{split}
\]
Combining the $\bx_{2\ell-1}\wedge\bx_{2\ell}$-factors
from $\bQ_k$ into the sets $A$ in the first sum (respectively $B$ in the second),
and noting that each resulting set then occurs $r-t$ (respectively $t+1$) times, we continue
\[
\begin{split}
    =&
    \sum\limits_{t=0}^{r-1}
        \frac{(-1)^t(r-t)!t!}{r!} \Big(
        \sum\limits_{\substack{|A|=r-t \\ |B|=t}}
        \quad
        \bigwedge\limits_{\ell\in A\cup B}\bx_{2\ell-1}\wedge\bx_{2\ell}
    \Big)
    \\&
    + \sum\limits_{t=0}^{r-2}
        \frac{(-1)^t(r-t-1)!(t+1)!}{r!} \Big(
        \sum\limits_{\substack{|A|=r-t-1 \\ |B|=t+1}}
        \quad
        \bigwedge\limits_{\ell\in A\cup B}\bx_{2\ell-1}\wedge\bx_{2\ell}
    \Big)
    \\=&
    \bx_1\wedge\cdots\wedge\bx_{k+1}+
    \sum\limits_{t=1}^{r-1}
        \frac{(-1)^t (r-t)! t!}{r!} \Big(
        \sum\limits_{\substack{|A|=r-t \\ |B|=t}}
        \;
        \bigwedge\limits_{\ell\in A\cup B}\bx_{2\ell-1}\wedge\bx_{2\ell}
    \Big)
    \\&
    + \sum\limits_{t=0}^{r-2}
        \frac{(-1)^t (r-t-1)!(t+1)!}{r!} \Big(
        \sum\limits_{\substack{|A|=r-t-1 \\ |B|=t+1}}
        \;
        \bigwedge\limits_{\ell\in A\cup B}\bx_{2\ell-1}\wedge\bx_{2\ell}
    \Big)
    \\=&
    \bx_1\wedge\cdots\wedge\bx_{k+1}+
    \sum\limits_{t=0}^{r-2}
        \frac{(-1)^{t+1}(r-t-1)! (t+1)!}{r!} \Big(
        \sum\limits_{\substack{|A|=r-t-1 \\ |B|=t+1}}
        \;
        \bigwedge\limits_{\ell\in A\cup B}\bx_{2\ell-1}\wedge\bx_{2\ell}
    \Big)
    \\&
    + \sum\limits_{t=0}^{r-2}
        \frac{(-1)^t (r-t-1)!(t+1)!}{r!} \Big(
        \sum\limits_{\substack{|A|=r-t-1 \\ |B|=t+1}}
        \quad
        \bigwedge\limits_{\ell\in A\cup B}\bx_{2\ell-1}\wedge\bx_{2\ell}
    \Big)
    \\=&
    \bx_1\wedge\cdots\wedge\bx_{k+1},
\end{split}
\]
completing the proof.
\end{proof}

\begin{proof}[Proof of Proposition \ref{prop:QsGenMinors}]
Choose any $(k+1)\times(k+1)$-minor of the $m\times 2k$ matrix $X = (x_{i,j})$.
Let $i_1<\cdots<i_{k+1}$ denote the rows and $j_1<\cdots<j_{k+1}$ the corresponding
columns, and let $\tau = \bx_{j_1}\wedge\cdots\wedge\bx_{j_{k+1}}$. Then
there is by Lemma \ref{lem:Wedge} an $\omega$ such that
$\bQ_k\wedge\omega = \tau$. In particular, the components of
this wedge product in terms of the basis elements $\be_{r_1}\wedge\cdots\wedge\be_{r_{k+1}}$
can be expressed algebraically in terms of the components of $\bQ_k$,
and the $\be_{i_1}\wedge\cdots\wedge\be_{i_{k+1}}$-component is
exactly the $(k+1)\times(k+1)$-minor in question.
\end{proof}

Combining Propositions \ref{prop:JiffQ} and \ref{prop:QsGenMinors} with the observations
above, we have proven the following.

\begin{theorem}
\label{thrm:IdealRelations}
Let $k, n \geq 1$. If $n \geq k$, the ideal of on-shell relations among the $x_{i,j}$
describing $\R[M_0]$ as an affine algebra is the real radical of the ideal generated by
the $Q_{i,j}$ for $1 \leq i < j \leq 2k$, and all $(k+1)\times(k+1)$-minors of $X$ are
contained in the ideal generated by the $Q_{i,j}$. If $n < k$, then the ideal of on-shell
relations is generated by both the $Q_{i,j}$ for $1 \leq i < j \leq 2k$ and the
$(n+1)\times(n+1)$-minors of $X$.
\end{theorem}

We suspect that the ideal generated by the $Q_{i,j}$ (along with the
$(n+1)\times(n+1)$-minors of $X$ when $n < k$) is in fact real, which we observe is
the case for small values of $k$ in Section \ref{sec:Computations}. However, the
description given by Theorem \ref{thrm:IdealRelations} will be sufficient for
our purposes.

Note that it is easy to see that the $r\times r$-minors of $X$ do not vanish on $Z$ when
when $r \leq k, n$, so that the relations given by minors are not superfluous in the cases
where $n < k$.
Specifically, setting each $\bq_\ell = \bp_\ell = \be_\ell$ to be a standard basis
vector for $\ell \leq r$ and $\bq_\ell = \bp_\ell = 0$ for $\ell > r$
yields an element of $V_{k,n}$ that is clearly in $Z$ yet on which an $r\times r$-minor
does not vanish (and other minors can be dealt with similarly); see
\cite[page 24]{LermanMontgomerySjamaar}. Hence we have the following.

\begin{corollary}
\label{cor:IsoCases}
Fix $k \geq 1$. The symplectic quotients $M_{0,k,n}$ and $M_{0,k,m}$ are $\Z^+$-graded regularly
symplectomorphic for each $m, n \geq k$. The symplectic quotients for $1\leq n\leq k$ are not
homeomorphic.
\end{corollary}
\begin{proof}
When $n \geq k$, the description given by Theorem \ref{thrm:IdealRelations} of the graded algebras
of regular functions does not depend on $n$, so the generators $x_{i,j}$ of the regular functions
of $M_{0,k,k}$ and $M_{0,k,n}$ obviously define isomorphic (graded) global charts.
Note that the ideal $\mathcal{J}_{k,k}$ of the moment map is real by Theorem \ref{thrm:LargeCases},
which implies that $\mathcal{J}_{k,k}$ generates the vanishing ideal of $Z_{k,k}$ in $\mathcal{C}^\infty(V_{k,k})$,
see \cite[Theorem 6.3 and Remark 6.4]{ArmsGotayJennings}. Therefore, the vanishing ideal of $Z_{k,k}$
in $\mathcal{C}^\infty(V_{k,k})^{\OO_k}$ is generated by the invariants in $\mathcal{J}_{k,k}$
in this case, and the same follows for the isomorphic cases $n \geq k$.
The inequalities defining the semialgebraic sets $V_{k,k}/\OO_k$ and $V_{k,n}/\OO_n$ are determined
in \cite[Example 0.8]{ProcesiSchwarz}, and coincide with the requirement that the Gram matrix $X$
is positive semidefinite (i.e. all minors are nonnegative). These obviously coincide for the full
quotients and hence the symplectic quotients, so this isomorphism defines a $\Z^+$-graded regular
diffeomorphism. The Poisson brackets of the generators $x_{i,j}$ are computed in
Equation \eqref{eq:PoissonBrackets}, and one observes that they do not depend on $n$;
therefore, this isomorphism is Poisson.

When $m < n \leq k$, the dimensions of $M_{0,m,k}$ and $M_{0,n,k}$ do not coincide.
This can be seen by applying the Kempf-Ness homeomorphism and considering the affine
GIT quotients $\C^{mk}\git \OO_m(\C)$ and $\C^{nk}\git \OO_n(\C)$. Each have ${k\choose 2}$
fundamental invariants given by the Euclidean inner products, the $(n+1)\times(n+1)$-minors
of the $k\times k$ Gram matrix of invariants are relations for both, but the
$(m+1)\times(m+1)$-minors of the Gram matrix are relations only for $\C^{mk}\git \OO_m(\C)$.
It follows that the Krull dimension of the ring $\C[\C^{mk}]^{\OO_m(\C)}$ is strictly lower
than that of $\C[\C^{nk}]^{\OO_n(\C)}$, from which the result follows.
\end{proof}


\section{Symplectic quotients by $\SO_n$ and orbifold criteria}
\label{sec:OrbitTypes}

In this section, we consider the relationship between the $\OO_n$- and $\SO_n$-symplectic
quotients and determine which of these quotients are $\Z^+$-graded regularly symplectomorphic
to orbifolds. We first note the following. Let $\by \in V_{k,n}$, and let $W_{\by}$ denote
the vector subspace of $\R^n$ spanned by $\{\by_1,\ldots,\by_{2k}\}$. Then the $\OO_n$-isotropy
group of $\by$ is trivial if and only if $\dim W_{\by} = n$. Specifically, if
$\dim W_{\by} \leq n-1$, then there is a reflection in $\OO_n$ through a hyperplane
containing $W_{\by}$, which therefore fixes $\by$; if $W_{\by} = \R^n$, then
any element of $\OO_n$ that fixes $\by$ must fix each element of $\R^n$ and hence is trivial.
Similarly, the $\SO_n$-isotropy group of $\by$ is trivial if and only if
$\dim W_{\by} \geq n-1$, as any subspace of $\R^n$ of codimension $2$ is fixed by
a nontrivial rotation, and conversely every element of $\SO_n$ fixes a subspace
of $\R^n$ of codimension at least $2$.

With this, we demonstrate the following.

\begin{theorem}
\label{thrm:SOnIso}
Let $k$ be a positive integer and let $n \geq k+1$. Then the symplectic quotients
$M_{0,k,n}$ and $M_{0,k,n}^{\SO}$ coincide.
\end{theorem}
\begin{proof}
Let $\by \in V_{k,n}$ such that $\bJ(\by) = 0$ and choose
$g\in\OO_n\smallsetminus\SO_n$. By Proposition \ref{prop:QsGenMinors},
the dimension of $W_{\by}$ is at most $k \leq n-1$. Then there is a reflection
$r \in \OO_n\smallsetminus\SO_n$ through a hyperplane of $\R^n$ that contains $W_{\by}$,
which therefore fixes $\by$. As $\OO_n/\SO_n \cong \Z_2$, there is an $h\in\SO_n$
such that $g = hr$, and then $g\by = hr\by = h\by$. It follows that the $\OO_n$-
and $\SO_n$-orbits of $\by$ coincide.
\end{proof}

For the corresponding complex symplectic quotients, a similar statement has been observed
in \cite[Appendix A.2.1]{TerpereauThesis}.

More algebraically, note that the invariants of the $\SO_n$-action on $V_{k,n}$
are given by the $\OO_n$-invariants along with the determinants
$\det(\by_{i_1},\ldots,\by_{i_n})$ of collections of $n$ vectors. By Proposition
\ref{prop:QsGenMinors}, these determinants all vanish. Hence these ``new" invariants
are forced to vanish, and the $\SO_n$-symplectic quotients are defined by identical
invariants and relations as the corresponding $\OO_n$-symplectic quotients.
Similarly, the Hilbert embeddings differ only by coordinates that vanish on the shell.

Note that the consequence of Theorem \ref{thrm:SOnIso} does not hold when $n\leq k$.
In this case, we may choose $\bq_\ell = \bp_\ell = \be_\ell$ for $\ell \leq n$ and
$\bq_\ell = \bp_\ell = 0$ for $\ell > n$. The result is a point that is easily seen to
be an element of the shell and have trivial $\OO_n$-isotropy so that the $\OO_n$- and
$\SO_n$-orbits do not coincide. Hence, when $n \leq k$, the $\OO_n$-quotient is a
branched double-cover of the corresponding $\SO_n$-quotient as explained in
\cite[Remark 5.4]{SjamaarLerman}. The branch points are those points where the $\SO_n$-
and $\OO_n$-orbits coincide, which as seen in the proof of Theorem \ref{thrm:SOnIso}
are exactly the points $\by$ such that $W_{\by}$ has dimension at most $n-1$.

The following consequence of these observations will be useful in the sequel.
Consider the embedding $\iota_{k-1}\co V_{k,k-1} \to V_{k,k}$ where each $\R^{k-1}$
factor is embedded into each $\R^k$ factor as the first $k-1$ coordinates. We consider $\OO_{k-1}$
as a subgroup of $\OO_k$ (and similarly $\mathfrak{o}_{k-1}$ as a subalgebra of
$\mathfrak{o}_k$) via the embedding $\iota_{k-1}$.

\begin{lemma}
\label{lem:Codim2Strat}
The symplectic quotient $M_{0,k,k}$ contains an orbit type stratum $S$ of real
codimension $2$. The closure $S^{cl}$ of $S$ (in the classical topology) is the
set of orbits of points in $\iota_{k-1}(Z_{k,k-1})$, and hence $\iota_{k-1}$
induces a $\Z^+$-graded regular symplectomorphism between $S^{cl}$ and
$M_{0,k,k-1}$, where the global chart for $S^{cl}$ is the restriction of that
for $M_{0,k,k}$. Moreover, the Poisson homomorphism $\R[M_{0,k,k}] \to \R[M_{0,k,k-1}]$
induced by $\iota_{k-1}$ is surjective.
\end{lemma}
\begin{proof}
Let $R = \{ \by\in V_{k,k} \mid \dim_\R W_{\by} = n-1 \}$. It is easy to see
that $R$ is exactly the orbit type corresponding to isotropy groups in $\OO_k$
generated by a single reflection. As $\OO_k$ acts transitively on
the Grassmannian of $(k-1)$-planes in $\R^k$, each orbit in $R$ intersects
$\iota_{k-1}(V_{k,k-1})$. Moreover, $R \cap \iota_{k-1}(V_{k,k-1})$ is equal to
the image under $\iota_{k-1}$ of the principal $\OO_{k-1}$-orbit type in
$V_{k,k-1}$ as explained above, and therefore $R \cap \iota_{k-1}(V_{k,k-1})$
is dense in $\iota_{k-1}(V_{k,k-1})$. Clearly, the $\OO_k$-orbits intersect
$\iota_{k-1}(V_{k,k-1})$ as $\OO_{k-1}$-orbits, and hence $\iota_{k-1}$ induces
a homeomorphism from the principal orbit type of $V_{k,k-1}/\OO_{k-1}$ onto
$R/\OO_k$.

From Equation \eqref{eq:DefJ}, one easily checks
that the moment map $\bJ_{k,k}$ of the $\OO_k$-action on $V_{k,k}$ restricts
to $\iota_{k-1}(V_{k,k-1})$ as the moment map $\bJ_{k,k-1}$ of the
$\OO_{k-1}$-action, i.e. $\bJ_{k,k-1} = \bJ_{k,k}\circ\iota_{k-1}$,
so that $Z_{k,k} \cap \iota_{k-1}(V_{k,k-1}) = \iota_{k-1}(Z_{k,k-1})$.
Hence $\iota_{k-1}$ induces a homeomorphism from the principal
orbit type stratum of $M_{0,k,k-1}$ onto
$(R\cap \iota_{k-1}(Z_{k,k-1}))/\OO_{k-1}$, which is homeomorphic
to $S := (R \cap Z_{k,k})/\OO_n$. Note that as $R$ is an $\OO_k$-orbit type
stratum of $V_{k,k}$, $S$ is by definition a stratum of the symplectic
stratified space $M_{0,k,k}$; see \cite[Theorem 2.1]{SjamaarLerman}.

Recall that $M_{0,k,k}$ is homeomorphic to the GIT quotient
$(\C^{k^2})\git \OO_k(\C)$ by the Kempf-Ness homeomorphism \cite{KempfNess,GWSkempfNess}.
By the Fundamental Theorem of Invariant Theory for $\OO_n$,
\cite[Sections 9.3 and 9.4]{PopovVinberg} and \cite[Sections 9 and 17]{Weyl},
it is easy to see that $\C[\C^{k^2}]^{\OO_k(\C)}$ is generated by ${k+1\choose 2}$
quadratic invariants with no relations, and hence that $(\C^{k^2})\git\OO_k(\C)$ is
complex ${k+1\choose 2}$-dimensional affine space. Similarly, $\C[\C^{k(k-1)}]^{\OO_k(\C)}$
is generated by ${k+1\choose 2}$ quadratic invariants with the single
relation that the determinant of the Gram matrix vanishes, and hence is
a complex hyperplane in $\C^{k+1\choose 2}$. Hence $S^{cl}$ has real codimension $2$
in $M_{0,k,k}$.

Finally, note that as $\bJ_{k,k-1} = \bJ_{k,k}\circ\iota_{k-1}$, the pullback
$\iota_{k-1}^\ast\co \R[V_{k,k}] \to \R[V_{k,k-1}]$ induces a homomorphism
$\R[M_{0,k,k}] \to \R[M_{0,k,k-1}]$. Recall that $\R[M_{0,k,k-1}]$ is generated by the ${2k+1\choose 2}$
scalar products, which all occur as pullbacks via $\iota_{k-1}$ of the scalar products that generate
$\R[M_{0,k,k}]$. It is then clear that the map $\R[M_{0,k,k}] \to \R[M_{0,k,k-1}]$
induced by $\iota_{k-1}$ is a surjective Poisson homomorphism that induces a Poisson
isomorphism $\R[S^{cl}]\to\R[M_{0,k,k-1}]$, where $\R[S^{cl}]$ is the quotient of
$\R[V_{k,k}]^{\OO_k}$ by the invariant part of the vanishing ideal of $R^{cl}$, i.e.
the algebra of regular functions corresponding to the restriction of the global chart
for $\R[M_{0,k,k}]$ to $S^{cl}$.
\end{proof}

Now we are ready to demonstrate that the only symplectic quotients under consideration
that are $\Z^+$-graded regularly symplectomorphic to a linear symplectic orbifold are the cases
where $k = 1$ or $n = 1$. By a \emph{linear symplectic orbifold},
we mean the symplectic orbifold given by the quotient of $\C^n$ by a finite subgroup of
$\U_n$. The argument below uses ingredients developed in \cite[Section 3]{HerbigSchwarzSeaton}
and \cite[Section 3]{HerbigSeatonImpos}.

\begin{theorem}
\label{thrm:OrbifoldCriteria}
The symplectic quotient $M_{0,k,n}$ (respectively $M_{0,k,n}^{\SO}$) is $\Z^+$-graded regularly
symplectomorphic to a linear symplectic orbifold if and only if $k = 1$ or $n=1$.
\end{theorem}
\begin{proof}
When $n = 1$, the group $\OO_n$ is finite, $\SO_n$ is trivial, and the moment map is zero,
so these case are quotients of affine space by a finite group and hence are in fact equal
to linear symplectic orbifolds (trivial orbifolds, i.e. affine space, in the case of $\SO_n$).
When $k = 1$, the symplectic quotients $M_{0,1,n}$ are all $\Z^+$-graded regularly isomorphic
to $M_{0,1,1}$ by Corollary \ref{cor:IsoCases}, which is the orbifold $\C/\Z_2$. Note that
these isomorphisms have been observed in \cite{BosGotayReduc,ArmsGotayJennings}; see Section
\ref{sec:Computations}.

Now assume $k, n \geq 2$, and consider $M_{0,k,n}^{\SO}$. When $n \leq k$, the corresponding
$\SO_n(\C)$-representation is $2$-large by Theorem \ref{thrm:LargeCases}.
Moreover, this representation is orthogonal with respect to the Euclidean inner product on $\C^n$
and hence is stable by \cite[Lemma 7.11]{GWSlifting}. Therefore, by
\cite[Theorem 1.1]{HerbigSchwarzSeaton}, the symplectic quotient is not symplectomorphic
(regularly or otherwise) to a linear symplectic orbifold. Similarly, when $n \leq k-1$,
the $\OO_n(\C)$ representation is $2$-large by Theorem \ref{thrm:LargeCases} and again stable by
\cite[Lemma 7.11]{GWSlifting} so that $M_{0,k,n}$ is not regularly symplectomorphic to
a linear symplectic orbifold by \cite[Theorem 1.3]{HerbigSchwarzSeaton}.

It remains only to consider the cases $M_{0,k,n}$ for $n \geq k$ and $M_{0,k,n}^{\SO}$
for $n \geq k+1$, which for fixed $k$ are all $\Z^+$-graded regularly symplectomorphic by
Corollary \ref{cor:IsoCases} and Theorem \ref{thrm:SOnIso}. Hence, it is sufficient to
show that $M_{0,k,k}$ is not $\Z^+$-graded regularly symplectomorphic to a linear symplectic
orbifold for each $k \geq 2$.

For the case $k=2$, the Hilbert series of $\R[M_{0,2,2}]$ is computed in Section
\ref{sec:Computations}, Equation \eqref{eq:k2n2Laurent}. If $M_{0,2,2}$ is $\Z^+$-graded
regularly symplectomorphic to the linear symplectic orbifold $U/H$, then the Hilbert
series of $\R[M_{0,2,2}]$ clearly coincides with that of $\R[U]^H$, and hence the
first coefficient $\gamma_0$ of the Laurent expansion of the Hilbert series at $t = 1$
is $1/|H|$; see \cite[Theorem 3.23]{PopovVinberg}. Because the first coefficient
of the Laurent expansion is not the reciprocal of an integer, we have a contradiction,
and $M_{0,2,2}$ is not $\Z^+$-graded regularly symplectomorphic to a linear orbifold;
see \cite[Section 7]{HerbigSeatonHSeries}.

Now, for some $k \geq 3$, assume for contradiction that there is $\Z^+$-graded regular
symplectomorphism $\chi\co M_{0,k,k} \to U/H$ with $H$ a finite group and $U$ a unitary
$H$-module; this implies that $\chi^\ast\co\R[U]^H\to\R[M_{0,k,k}]$ is a $\Z^+$-graded
Poisson isomorphism. We will demonstrate that this implies a $\Z^+$-graded regular symplectomorphism
between $M_{0,k,k-1}$ and a linear symplectic suborbifold of $U/H$, which contradicts
\cite[Theorem 1.3]{HerbigSchwarzSeaton} as explained above.

By Lemma \ref{lem:Codim2Strat}, $M_{0,k,k}$ has a stratum $S$ whose closure
$S^{cl}$ is homeomorphic $M_{0,k,k-1}$. As the Poisson algebra of smooth functions
determines the connected components of the strata of a symplectic quotient by
\cite[Proposition 3.3]{SjamaarLerman}, $U$ must also have a real codimension $2$ orbit
type stratum $T$ with $\chi(S) = T$, and hence $\chi(S^{cl}) = T^{cl}$. Say
$T$ is the orbit type corresponding to $K \leq H$, let $L = N_H(K)/K$, and then
$T^{cl}$ is homeomorphic to $U^K/L$.

By \cite[Theorem 2]{HerbigSeatonImpos}, we have that $\chi$ restricts to a
$\Z^+$-graded regular
symplectomorphism $\chi|_{S^{cl}}\co S^{cl} \to T^{cl}$, where the global charts
for these spaces are the restrictions of those for $M_{0,k,k}$ and $U/H$,
respectively.  By Lemma \ref{lem:Codim2Strat}, the map
$\R[M_{0,k,k}]\to\R[M_{0,k,k-1}]$ is surjective, so the restricted global chart on
$S^{cl}$ is isomorphic to the usual global chart of $\R[M_{0,k,k-1}]$ given by the
$\OO_{k-1}$-invariants. Hence, we have a Poisson isomorphism
$\chi|_{S^{cl}}^\ast$ of $\R[M_{0,k,k-1}]$ into a subalgebra $\mathcal{A}$ of $\R[U^K]^L$
corresponding to the homeomorphism
$M_{0,k,k-1} \simeq S^{cl} \simeq T^{cl} \simeq U^K/L$; in particular,
$\mathcal{A}$ separates points in $U^K/L$. To show that this is in fact a
$\Z^+$-graded regular symplectomorphism onto $U^K/L$ with global chart induced
from $\R[U^K]^L$, it is sufficient to show that $\chi|_{S^{cl}}^\ast$ is surjective.
The proof follows \cite[Theorem 4]{HerbigSeatonImpos}; we summarize the argument here
but refer the reader to that reference for more details.

Tensoring with $\C$ yields isomorphisms
\[
    (\chi^\ast)^{\C}\co \C[V_{k,k}\times V_{k,k}^\ast]^{\OO_k(\C)}/(\mathcal{J}_{k,k}^{\C})^{\OO_k(\C)}
        \to\C[U\times U^\ast]^H
\]
and
\[
    (\chi|_{S^{cl}}^\ast)^\C\co \C[V_{k,k-1}\times V_{k,k-1}^\ast]^{\OO_{k-1}(\C)}/
        (\mathcal{J}_{k,k-1}^{\C})^{\OO_{k-1}(\C)}\to\mathcal{A}^\C,
\]
where we recall that $V_{k,k}^\ast$ denotes the dual representation and $\mathcal{J}_{k,k}^{\C}$,
$\mathcal{A}^{\C}$, etc. the complexifications. As the variety associated to $\mathcal{J}_{k,k-1}^{\C}$ is
normal by Theorem \ref{thrm:LargeCases}, the variety in
$\C[V_{k,k-1}\times V_{k,k-1}^\ast]^{\OO_{k-1}(\C)}$ associated to $(\mathcal{J}_{k,k-1}^{\C})^{\OO_{k-1}(\C)}$
is as well normal by \cite[Theorem 3.16]{PopovVinberg}, hence $\mathcal{A}^\C$ is integrally closed. Moreover,
it can be seen as in the proof of \cite[Theorem 4]{HerbigSeatonImpos} that $\mathcal{A}^{\C}$ is a
separating algebra in the sense of \cite[Definition 2.3.8]{DerskenKemperBook}. Specifically,
the map $(U\times U^\ast)^K/L \to (U\times U^\ast)/H$ induced by the embedding
$(U\times U^\ast)^K\to U\times U^\ast$ is injective as it is injective on the open dense
set of points with isotropy group $K$ and hence a birational map between normal varieties.
Then as $\C[U\times U^\ast]^H$ separates orbits in $(U\times U^\ast)/H$, as the map
$(\chi^\ast)^{\C}$ is an isomorphism, and as $\mathcal{A}$ separates points in $U^K/L$,
it follows that $\mathcal{A}^\C$ is a separating algebra. But then by
\cite[Theorem 2.3.12]{DerskenKemperBook}, $\C[(U\times U^\ast)^K]^L$ is equal to the normalization
of $\mathcal{A}^{\C}$, so that as $\mathcal{A}^{\C}$ is integrally closed,
$\C[(U\times U^\ast)^K]^L = \mathcal{A}^{\C}$. Hence $(\chi|_{S^{cl}}^\ast)^\C$ is surjective so
that $\chi|_{S^{cl}}^\ast$ is surjective, completing the proof.
\end{proof}


\section{Rational singularities}
\label{sec:RatSing}

In this section, we prove the following.

\begin{theorem}
\label{thrm:RationalSingMoment}
Suppose $1 \leq n \leq k$. Then the real variety given by the Zariski closure of $Z_{k,n}$
has rational singularities.
\end{theorem}

Note that the cases $1 \leq n \leq k$ are exactly those such that $Z_{k,n}$ is normal
by Theorem \ref{thrm:LargeCases}.

\begin{proof}
As $n \leq k$, the scheme $\mathcal{J}_\C$ is a normal, reduced, irreducible complete
intersection by Theorem \ref{thrm:LargeCases}({\it i}). Counting invariants and relations,
the Hilbert Series is given by
\[
    \Hilb(x)    =   \frac{ (1 - x^2)^{n\choose 2} }{ (1 - x)^{2kn} },
\]
so the $a$-invariant is $n(n - 2k - 1)$. Moreover, the (real) vanishing ideal of $Z_{k,n}$ is equal
to $\mathcal{J}_{k,n}$, again by Theorem \ref{thrm:LargeCases}({\it i}).

When $n = 1$, the ideal $\mathcal{J}_{1,k} = 0$ and the result is trivial.
Suppose $n = 2$. Then $\mathcal{J}_{2,k}$ is generated by the single element
\[
    J_{1,2} = \sum\limits_{\ell=1}^k q_{\ell,1} p_{\ell,2} - q_{\ell,2} p_{\ell,1}.
\]
By the theorem of Flenner \cite{FlennerRational} and Watanabe \cite{WatanabeRational}
(see also \cite[Theorem 9.2]{HunekeTightClosure}), as $\R[V_{k,n}]/\mathcal{J}_{k,n}$ is
Cohen-Macaulay and normal with negative $a$-invariant, it is sufficient to show that
$(\R[V_{k,n}]/\mathcal{J}_{k,n})_P$ is rational for all primes $P$ not equal to the
homogenous maximal ideal. However, by a simple computation, adjoining an inverse to any of the variables
yields a localization of a polynomial ring. Therefore, the origin is an isolated
singularity, and $\R[V_{k,n}]/\mathcal{J}_{k,n}$ has rational singularities. Note that
this establishes the claim for $k \leq 2$.

Now, fix $n$, and assume for induction that $\R[V_{m,k}]/\mathcal{J}_{m,k}$ has
rational singularities for any $k \geq 3$ and $2\leq m < n \leq k$. We fix a value of
$k \geq 3$ and prove that $\R[V_{k,n}]/\mathcal{J}_{k,n}$ has rational singularities.

Note that $\mathcal{J}_{k,n}$ is invariant under any isomorphism given by a permutation of
the pairs $(\bq_\ell, \bp_\ell)$, the permutation $\bq_\ell \leftrightarrow \bp_\ell$, or
induced by an element of $\OO_n$.
Any nonzero point in $V_{k,n}$ can be mapped by these actions to a point such that $\bq_1\neq 0$.
Similarly, any point with $\bq_1\neq 0$ is in the $\OO_n$-orbit of a point such that
$q_{1,\alpha} \neq 0$ for $\alpha = 1,\ldots,n$. Hence, it is sufficient to restrict to
neighborhoods of such a point. Localizing near a point with $q_{1,\alpha} \neq 0$ for
$\alpha = 1,\ldots,n$, we adjoin inverses to $q_{1,\alpha}$ for each $\alpha$. Then for
$\beta = 2,\ldots,n$, the $J_{1,\beta}$ can be
expressed as
\begin{equation}
\label{eq:EliminatedVars}
    p_{1,\beta}     =       q_{1,1}^{-1} \left(q_{1,\beta} p_{1,1}
                        + \sum\limits_{\ell=2}^{k} q_{\ell,\beta} p_{\ell,1} - q_{\ell,1} p_{\ell,\beta}\right),
                \quad\quad \beta = 2, \ldots, n,
\end{equation}
eliminating the variables $p_{1,\beta}$ for $\beta \neq 1$.
Then each $J_{\alpha,\beta}$ with $2 \leq \alpha < \beta$ becomes
\[
\begin{split}
    J_{\alpha,\beta} =&  q_{1,\alpha} p_{1,\beta} - q_{1,\beta} p_{1,\alpha}
                + \sum\limits_{\ell=2}^{k} q_{\ell,\alpha} p_{\ell,\beta} - q_{\ell,\beta} p_{\ell,\alpha}
    \\=&    q_{1,\alpha} \left(q_{1,\beta} \frac{p_{1,1}}{q_{1,1}}
                + \sum\limits_{\ell=2}^{k} q_{\ell,\beta}
                   \frac{p_{\ell,1}}{q_{1,1}} - \frac{q_{\ell,1}}{q_{1,1}} p_{\ell,\beta}\right)
    \\&
                - q_{1,\beta} \left(q_{1,\alpha} \frac{p_{1,1}}{q_{1,1}}
                + \sum\limits_{\ell=2}^{k} q_{\ell,\alpha}
                    \frac{p_{\ell,1}}{q_{1,1}} - \frac{q_{\ell,1}}{q_{1,1}} p_{\ell,\alpha}\right)
                + \sum\limits_{\ell=2}^{k} q_{\ell,\alpha} p_{\ell,\beta} - q_{\ell,\beta} p_{\ell,\alpha} .
\end{split}
\]
Therefore, we can express
\begin{equation}
\begin{split}
\label{eq:RationalRelations}
    &           \frac{J_{\alpha,\beta}}{q_{1,\alpha}q_{1,\beta}}
    \\
    =&           \sum\limits_{\ell=2}^{k}
                    \frac{q_{\ell,\beta}}{q_{1,\beta}} \frac{p_{\ell,1}}{q_{1,1}}
                        - \frac{q_{\ell,1}}{q_{1,1}} \frac{p_{\ell,\beta}}{q_{1,\beta}}
                    +
                    \frac{q_{\ell,1}}{q_{1,1}} \frac{p_{\ell,\alpha}}{q_{1,\alpha}}
                        - \frac{q_{\ell,\alpha}}{q_{1,\alpha}} \frac{p_{\ell,1}}{q_{1,1}}
                    +
                    \frac{q_{\ell,\alpha}}{q_{1,\alpha}} \frac{p_{\ell,\beta}}{q_{1,\beta}}
                        - \frac{q_{\ell,\beta}}{q_{1,\beta}} \frac{p_{\ell,\alpha}}{q_{1,\alpha}}
    \\
    =&          \sum\limits_{\ell=2}^{k}
                    \left( \frac{q_{\ell,\alpha}}{q_{1,\alpha}} - \frac{q_{\ell,1}}{q_{1,1}} \right)
                    \left( \frac{p_{\ell,\beta}}{q_{1,\beta}} - \frac{p_{\ell,1}}{q_{1,1}} \right)
                    -
                    \left( \frac{q_{\ell,\beta}}{q_{1,\beta}} - \frac{q_{\ell,1}}{q_{1,1}}\right)
                    \left(\frac{p_{\ell,\alpha}}{q_{1,\alpha}} - \frac{p_{\ell,1}}{q_{1,1}} \right).
\end{split}
\end{equation}

Now, consider the algebra
\[
    \R[\bx, \by] :=
    \R[x_{\ell,\beta}, y_{\ell,\beta} \mid \ell = 2,\ldots,k; \quad \beta = 2,\ldots, n]
\]
and ideal $\mathcal{J}^\prime$ generated by
\begin{equation}
\label{eq:DefJPrime}
    J_{\alpha,\beta}^\prime
    =   \sum\limits_{\ell=2}^k x_{\ell,\alpha} y_{\ell,\beta} - x_{\ell,\beta} y_{\ell,\alpha}
\end{equation}
for $2 \leq \alpha < \beta \leq n$. Then $\R[\bx,\by]/\mathcal{J}^\prime$
is obviously isomorphic to $\R[V_{n-1,k-1}]/\mathcal{J}_{n-1,k-1}$ by simply relabeling
variables. Hence, by the inductive hypothesis,
$\R[\bx,\by]/\mathcal{J}^\prime$ has rational singularities and in
particular is normal.
We adjoin the additional variables $z_{1,\alpha}$ for $\alpha = 1,\ldots,n$; $z_{\ell,1}$ for
$\ell = 2,\ldots,k$; and $w_{\ell,1}$ for $\ell = 1,\ldots,k$. Let
\[
\begin{split}
    W   =&      \R^{2(k-1)(n-1)}\oplus\R^{n+k-1}\oplus \R^{k}
        \\=&    \{(\bx, \by, \bz, \bw) |
                \bx,\by \in \R^{2(k-1)(n-1)},
                \bz\in \R^{n+k-1}, \bw\in \R^{k}\}
\end{split}
\]
be the vector space with these coordinates, let $Y$ denote the affine variety in $W$
with ideal $\mathcal{J}^\prime$, and then $Y$ is a normal variety with
rational singularities.

Let $\mathcal{O}$ be the open subset of $V_{k,n}$ defined by $q_{1,\alpha}\neq 0$ for
$\alpha = 1,\ldots,n$, let $\mathcal{U}$ be the open subset of $W$ defined by $z_{1,\alpha}\neq 0$ for
$\alpha = 1,\ldots,n$, and then it is easy to see that
of $\mathcal{O}\cap Z$ and $\mathcal{U}\cap Y$ are birationally equivalent.
As $\mathcal{U}\cap Y$ is a localization of the normal variety $Y$ and hence a normal variety,
it follows from Zariski's main theorem that  $\mathcal{O}\cap Z$ and $\mathcal{U}\cap Y$ are
isomorphic. As $Y$ and hence $\mathcal{U}\cap Y$ has
rational singularities, it follows that $\mathcal{O}\cap Z$ has rational singularities.

Again, $\R[V_{k,n}]/\mathcal{J}_{k,n}$ is Cohen-Macaulay, normal, and has negative
$a$-invariant. We again apply the theorem of Flenner \cite{FlennerRational} and Watanabe
\cite{WatanabeRational} to conclude that $\R[V_{k,n}]/\mathcal{J}_{k,n}$ has rational
singularities near the origin. Then by induction, we are done.
\end{proof}

By Theorem \ref{thrm:RationalSingMoment}, when $n \leq k$, the spectrum of the
ring $\R[V_{k,n}]/\mathcal{I}_Z$ has rational singularities. It then follows from
Boutot's Theorem \cite{Boutot} that the spectrum of $\R[M_0]$ and $\R[M_0^{\SO}]$ have
only rational singularities as well. If $n > k$, then $\R[M_{0,k,n}]$ is isomorphic to
$\R[M_{0,k,k}]$ by Corollary \ref{cor:IsoCases}, and $\R[M_{0,k,n}^{\SO}]$ is isomorphic to
$\R[M_{0,k,k}]$ by Theorem \ref{thrm:SOnIso}. Hence we have the following.

\begin{corollary}
\label{cor:RationalSing}
For each $k, n\geq 1$, the (Zariski closure of the) symplectic quotients $M_0$ and $M_0^{\SO}$
have rational singularities. In particular, they are normal, and the rings
$\R[M_0]$ and $\R[M_0^{\SO}]$ of regular functions are Cohen-Macaulay.
\end{corollary}

Note that Corollary \ref{cor:RationalSing} is also a consequence of the results of
\cite{TerpereauThesis} and \cite{BeauvilleSympSing}. In particular, by
\cite[Propisiton A.2.1 and Lemma A.2.2]{TerpereauThesis}, the (complex) symplectic quotients
$(\bJ^\C)^{-1}(0)\git\OO_n(\C)$ and $(\bJ^\C)^{-1}(0)\git\SO_n(\C)$
are so-called symplectic varieties. Then by \cite[Proposition 1.3]{BeauvilleSympSing},
symplectic varieties are Gorenstein and have rational singularities, properties
that are easily seen to descend from the complexifications to the real quotients.
Note that it does not follow that the symplectic quotients $M_0$ and $M_0^{\SO}$
are \emph{graded} Gorenstein, which appears to have been established
only for (complex) symplectic quotients by tori and those that admit a symplectic resolution;
see \cite[Sections 3.2 and 3.4]{McGertyNevins}. In particular, the condition
\cite[Assumption 3.2 (2a)]{McGertyNevins} corresponds to the symplectic quotient
being graded Gorenstein. We consider this condition in the following section.


\section{Explicit computations}
\label{sec:Computations}

For small values of $k$, we have computed the on-shell relations that can be used
to describe $\R[M_0]$ and $\R[M_0^{\SO}]$ as an affine algebra explicitly using elimination.
These computations have motivated many of the findings in this paper and have allowed us to
verify that for small $k$, the ideal described by Theorem \ref{thrm:IdealRelations} is in fact
real radical, i.e. there are no other on-shell relations. We have also used these
computations to test whether the algebra $\R[M_0]$ is graded Gorenstein, which
we conjecture is true in general.

The method is as follows; see \cite[Chapters 2 and 3]{CoxLittleOsheaIVA} for background.
Let $\R[\bq,\bp,\bY]$ be the polynomial ring in
the variables $q_{\ell,\alpha}$, $p_{\ell,\alpha}$, and $Y_{i,j}$ where $1\leq \ell \leq k$,
$1\leq \alpha \leq n$, and $1\leq i \leq j \leq 2k$; the $q_{\ell,\alpha}$ and $p_{\ell,\alpha}$
have degree $1$ and the $Y_{i,j}$ have degree $2$. Using the software package \emph{Mathematica}
\cite{Mathematica} and a monomial order that is an elimination order for the variables
$q_{\ell,\alpha}$ and $p_{\ell,\alpha}$, we compute a Gr\"{o}bner basis for the ideal of $\R[\bq,\bp,\bY]$
generated both by the components of the moment map as well as the functions $Y_{i,j} - x_{i,j}$ where the
$x_{i,j}$ are the invariant scalar products described in Section \ref{sec:Setup}.
Computing the elimination ideal corresponding to eliminating the variables
$q_{\ell,\alpha}$ and $p_{\ell,\alpha}$ yields the ideal $\mathcal{R} = \mathcal{R}_{k,n}$
of relations in the intersection of the algebra
$\R[V_{k,n}]^{\OO_n} = \R[x_{i,j}| 1\leq i \leq j \leq 2k] =: \R[X]$
and the ideal $\mathcal{J} = \mathcal{J}_{k,n}$ generated by the moment map.
When $n \leq k+1$, the ideal $\mathcal{J}_{k,n}$ is real radical by Theorem \ref{thrm:LargeCases}
so that $\mathcal{R}_{k,n}$ is equal to the ideal of $\R[M_0]$ in $\R[X]$.
The computations to compute the ideal $\mathcal{R}_{k,n}^{\SO}$ of $\R[M_{0,k,n}^{\SO}]$ are similar;
the only difference is that we include the invariant determinants along with the $x_{i,j}$ and additional
corresponding variables in $\bY$.

Let $\mathcal{Q}=\mathcal{Q}_{k,n}$ denote the ideal described in Theorem
\ref{thrm:IdealRelations}, i.e. $\mathcal{Q}_{k,n}$ is generated by the $Q_{i,j}$ when $k\geq n$
and the $Q_{i,j}$ along with the $(n+1)\times(n+1)$-minors of $X$ when $n < k$, and
recall that by the same theorem, the ideal of relations in $\R[M_{0,k,n}]$ is the
real radical of $\mathcal{Q}_{k,n}$.
We then have that when $n \leq k+1$, as $\mathcal{R}_{k,n}$ is the ideal of relations of $\R[M_{0,k,n}]$,
$\mathcal{R}_{k,n}$ is the real radical of $\mathcal{Q}_{k,n}$. For fixed $k$, the cases $n \geq k$ yield
isomorphic $\R[M_0]$ by Corollary \ref{cor:IsoCases}, and the computation of $\mathcal{R}_{k,n}$ can be
used to verify this fact. That the ideal $\mathcal{R}_{k,n}$ is real radical in these cases follows from
this isomorphism.

Experimentally, we have found that the elimination is much faster using the monomial order
\[
    q_{1,1} \succ p_{1,1} \succ q_{1,2} \succ p_{1,2} \succ\cdots\succ q_{1,n} \succ p_{1,n}
    \succ q_{2,1}\succ p_{2,1} \succ\cdots\succ q_{k,n} \succ p_{k,n}
\]
on the elimination variables; the ordering on the $x_{i,j}$ does not appear to affect
the time of the computations significantly. As an example, on a Windows 7 desktop computer
running \emph{Mathematica} 8.0
with an i5 Core Processor, 4 GB or RAM, and processing priority set to \emph{High},
the Gr\"{o}bner basis computation for $k = 2$ and $n = 4$ using Lexicographic order took 403 seconds,
while it took less than one second with the above monomial order; similar extreme differences
were noted in other cases. For larger cases, we also used the standard \emph{Mathematica}
\texttt{EliminationOrder} with the option \texttt{Sort -> True}, which allows the algorithm to
reorder the variables.

Once the ideal of relations $\mathcal{R}$ of $\R[M_0]$ (and $\mathcal{R}^{\SO}$ of $\R[M_0^{\SO}]$)
has been determined, we compute the Hilbert series
of $\mathcal{R}$ and $\mathcal{Q}$ using \emph{Macaulay2} \cite{M2}. Note that $\mathcal{R}$ contains the
ideal $\mathcal{Q}$ by construction and the above observations. Hence, equality of these ideals can be
verified by observing that their Hilbert series coincide. For smaller cases, we have also used the
ideal membership and polynomial division features of \emph{Macaulay2} to express each element of the
Gr\"{o}bner basis for $\mathcal{R}_{k,n}$ in terms of the generators of $\mathcal{Q}_{k,n}$ explicitly,
though most of these expressions are too large to report here.

We check that each Hilbert series satisfies the functional equation
\begin{equation}
\label{eq:GradedGorenCond}
    \Hilb(t^{-1})   =   (-1)^d t^{-a}\Hilb(t)
\end{equation}
where $d$ is the Krull dimension and $a$ is the $a$-invariant. As the corresponding quotient rings are
Cohen-Macaulay integral domains by Corollary \ref{cor:RationalSing}, this functional equation is equivalent
to the rings being Gorenstein by \cite[Theorem 4.4]{StanleyHilbert}. In fact, we observe in each case
that $d = -a$, and hence that the quotient ring is graded Gorenstein; see
\cite[Section 3.7.3]{DerskenKemperBook}. When $k = 2$, these computations were considered in
\cite[Section 8.3.1]{HerbigSeatonHSeries}, and the relation between the coefficients of the Laurent
expansion of the Hilbert series at $t=1$ considered there and the graded Gorenstein condition
is explained by \cite[Corollary 1.8]{HerbigHerdenSeaton}.

Note that the cases corresponding to $n = 1$ are orbifolds in a strict sense: the moment map is $0$,
and the group $\OO_1 = \Z_2$ is finite. Hence in these cases, the fact that $\R[M_0]$ has rational singularities
is a direct consequence of Boutot's Theorem \cite{Boutot}, and the fact that $\R[M_0]$ is graded
Gorenstein follows from Watanabe's Theorem \cite{WatanabeGor1,WatanabeGor2}. We have presented these
cases for completeness, but they are particularly easy to compute. The corresponding $\SO_1$-quotients
are affine spaces as $\SO_1$ is trivial.

To summarize the experimental results discussed below, we were able to compute the ideals
$\mathcal{R}_{k,n}$ and $\mathcal{Q}_{k,n}$ for $(k,n)=(1,1), (2,1), (2,2), (3,1), (3,2), (3,3), (4,1)$,
and $(4,2)$;
for the case $(4,4)$, we
were only able to compute $\mathcal{Q}_{k,n}$. We were able to compute $\mathcal{R}_{k,n}^{\SO}$
only for the cases $(k,n) = (2,2)$ and $(3,2)$. In each case, the ring of regular functions
is graded Gorenstein, and $\mathcal{R}_{k,n} = \mathcal{Q}_{k,n}$ in each case where we were
able to compute both. We also observe some interesting patterns, e.g. the coefficients of $t^2$
in the numerator appears to be a square number when $k = n$ for both the $\OO_n$- and
$\SO_n$-quotients.

Many of these computations were completed in a few seconds on a laptop with
a 2.4 GHz two-core processor (utilizing only one core for Gr\"{o}bner basis computations)
and 4 GB RAM. When longer computation times were encountered or more sophisticated
equipment was required, it is noted below.


\subsection{$k = 1$}
When $k = 1$ and $n = 1$, the Gr\"{o}bner basis for $\mathcal{R}_{1,1}$ consists of the single
relation $x_{1,1} x_{2,2} - x_{1,2}^2$; this is equal to $Q_{1,1}$ so that
$\mathcal{R}_{1,1} = \mathcal{Q}_{1,1}$ is obvious in this case. All other values of $n$ yield
the same result by Corollary \ref{cor:IsoCases}. The Hilbert series of $\R[M_{0,1,n}]$ is given
by
\[
    \Hilb_{\R[M_{0,1,n}]}(t)    =   \frac{1 + t^2}{(1 - t^2)^2},
\]
which satisfies Equation \eqref{eq:GradedGorenCond} with $d = -a = 2$. Hence for each $n$,
the symplectic quotient
$M_{0,1,n}$ is $\Z^+$-graded regularly symplectomorphic to the orbifold $\C/\Z_2$ where $\Z_2$
acts as multiplication by $-1$, which has already been observed in
\cite{BosGotayReduc,ArmsGotayJennings}.

In this case, there are no new $\SO_n$-quotients, as the case $n=1$ is affine space, and
the cases $n \geq 2$ are all isomorphic to the corresponding $\OO_n$-quotients, described
above, by Theorem \ref{thrm:SOnIso}


\subsection{$k = 2$}
When $k = 2$ and $n = 1$, we again have $\bJ = 0$. The Gr\"{o}bner basis of
$\mathcal{R}_{2,1}$ contains $20$ elements, each easily recognized as a $2\times 2$-minor
of the Gram matrix $X$. The Hilbert series is
\[
    \Hilb_{\R[M_{0,2,1}]}(t) =\frac{1 + 6t^2 + t^4}{(1 - t^2)^4},
\]
which satisfies Equation \eqref{eq:GradedGorenCond} with $d = -a = 2$ and coincides
with the Hilbert series of the ideal $\mathcal{Q}_{2,1}$ so that $\mathcal{Q}_{2,1}$ is real
radical.

When $k = 2$ and $n = 2$, the Gr\"{o}bner basis of $\mathcal{R}_{2,2}$ contains the
$9$ elements
\[
\begin{split}
    -x_{24}^2 x_{33} + 2 x_{23} x_{24} x_{34} - x_{22} x_{34}^2 - & x_{23}^2 x_{44} +  x_{22} x_{33} x_{44}
        \\=&
        x_{24} Q_{23} - x_{23} Q_{24} + x_{22} Q_{34},
    \\
    -x_{14} x_{23} + x_{13} x_{24} - x_{34}^2 + x_{33} x_{44}
        =&
        Q_{34},
    \\
    -x_{14} x_{24} x_{33} + 2 x_{14} x_{23} x_{34} - x_{12} x_{34}^2 + x_{34}^3
        - &x_{13} x_{23} x_{44} + x_{12} x_{33} x_{44} - x_{33} x_{34} x_{44}
        \\=&
        x_{24} Q_{13} -x_{23} Q_{14} + (x_{12} - x_{34}) Q_{34},
    \\
    -x_{14} x_{22} + x_{12} x_{24} - x_{24} x_{34} + x_{23} x_{44}
        =&
        Q_{24},
    \\
    -x_{13} x_{22} + x_{12} x_{23} - x_{24} x_{33} + x_{23} x_{34}
        =&
        Q_{23},
    \\
    -x_{14}^2 x_{33} + 2 x_{13} x_{14} x_{34} - x_{11} x_{34}^2 - & x_{13}^2 x_{44} + x_{11} x_{33} x_{44}
        \\=&
        x_{14} Q_{13} - x_{13} Q_{14} + x_{11} Q_{34},
    \\
    -x_{12} x_{14} + x_{11} x_{24} - x_{14} x_{34} + x_{13} x_{44}
        =&
        Q_{14},
    \\
    -x_{12} x_{13} + x_{11} x_{23} - x_{14} x_{33} + x_{13} x_{34}
        =&
        Q_{13},
    \\
    -x_{12}^2 + x_{11} x_{22} + x_{34}^2 - x_{33} x_{44}
        =&
        Q_{12} - Q_{34}.
\end{split}
\]
Hence, $\mathcal{R}_{2,2} = \mathcal{Q}_{2,2}$. This can also be seen by computing
that for both $\mathcal{R}_{2,2}$ and $\mathcal{Q}_{2,2}$, the Hilbert series is
\[
    \Hilb_{\R[M_{0,2,2}]}(t) =\frac{1 + 4 t^2 + 4 t^4 + t^6}{(1 - t^2)^6},
\]
which satisfies Equation \eqref{eq:GradedGorenCond} with $d = -a = 6$.
In particular, the Laurent expansion of the Hilbert series at $t = 1$ begins
\begin{equation}
\label{eq:k2n2Laurent}
    \Hilb_{\R[M_{0,2,2}]}(t) = \frac{5}{32}(1-t)^{-6}
        + \frac{11}{128}(1-t)^{-4} + \frac{11}{128}(1-t)^{-3} + \cdots,
\end{equation}
which is used in the proof of Theorem \ref{thrm:OrbifoldCriteria}.
For each $n > 2$, $\R[M_{0,2,n}]$ is isomorphic to $\R[M_{0,2,2}]$ by
Corollary \ref{cor:IsoCases}.

The $\SO_n$-symplectic quotient is affine space when $n = 1$ and isomorphic to the corresponding
$\OO_n$ quotients when $n\geq 3$, so the only $\SO_n$-quotient not already described is $M_{0,2,2}^{\SO}$.
In this case, the Hilbert series is
\[
    \Hilb_{\R[M_{0,2,2}^{\SO}]}(t) =\frac{1 + 9 t^2 + 9 t^4 + t^6}{(1 - t^2)^6},
\]
which satisfies Equation \eqref{eq:GradedGorenCond} with $d = -a = 6$.


\subsection{$k = 3$}

When $k = 3$ and $n = 1$, the Gr\"{o}bner basis of $\mathcal{R}_{3,1}$ contains
$105$ elements, each easily recognized as a $2\times 2$-minor of the Gram matrix $X$.
The Hilbert series is
\[
    \Hilb_{\R[M_{0,3,1}]}(t) =\frac{1 + 15t^2 + 15t^4 + t^6}{(1 - t^2)^6},
\]
which satisfies Equation \eqref{eq:GradedGorenCond} with $d = -a = 6$ and coincides with
the Hilbert series of the ideal $\mathcal{Q}_{3,1}$, hence $\mathcal{Q}_{3,1}$ is real radical.

When $k = 3$ and $n = 2$, the Hilbert series is
\[
    \Hilb_{\R[M_{0,3,2}]}(t) =\frac{1 + 11 t^2 + 51 t^4 + 51 t^6 + 11 t^8 + t^{10}}{(1 - t^2)^{10}},
\]
which satisfies Equation \eqref{eq:GradedGorenCond} with $d = -a = 10$ and coincides with
the Hilbert series of the ideal $\mathcal{Q}_{3,2}$ so that $\mathcal{Q}_{3,2}$ is real radical.

When $k = 3$ and $n = 3$, the Hilbert series is
\[
    \Hilb_{\R[M_{0,3,3}]}(t)    =
    \frac{1 + 9 t^2 + 30 t^4 + 44 t^6 + 30 t^8 + 9 t^{10} + t^{12}}{(1 - t^2)^{12}},
\]
which satisfies Equation \eqref{eq:GradedGorenCond} with $d = -a = 12$ and coincides with
the Hilbert series of the ideal $\mathcal{Q}_{3,3}$ so that this ideal is real radical.
In this case, $\mathcal{R}_{3,3}$ took about $96$ minutes to compute on a laptop.

The $\SO_n$-symplectic quotient is affine space when $n = 1$ and isomorphic to $\OO_n$-quotients
when $n \geq 4$ by Theorem \ref{thrm:SOnIso}. When $n=2$, the computation of $\mathcal{R}_{3,2}^{\SO}$
takes about $7$ minutes on a laptop, and is generated by $267$ relations. The Hilbert series is
\[
    \Hilb_{\R[M_{0,3,2}^{\SO}]}(t) =\frac{1 + 25 t^2 + 100 t^4 + 100 t^6 + 25 t^8 + t^{10}}{(1 - t^2)^{10}},
\]
which satisfies Equation \eqref{eq:GradedGorenCond} with $d = -a = 10$.
We were not able to compute $\mathcal{R}_{3,3}^{\SO}$; using the option \texttt{Sort->True}
and elimination order on \emph{Mathematica} on a 244GB RAM machine, the computation had not completed
in 50 hours.


\subsection{$k = 4$}

When $n = 1$, the Gr\"{o}bner basis of $\mathcal{R}_{4,1}$ is computed in under $2$ minutes on a
laptop and contains $336$ elements, each easily recognized as a $2\times 2$ minor of the Gram matrix $X$.
The Hilbert series is
\[
    \Hilb_{\R[M_{0,4,1}]}(t) =\frac{1 + 28 t^2 + 70 t^4 + 28 t^6 + t^8}{(1 - t^2)^8},
\]
which satisfies Equation \eqref{eq:GradedGorenCond} with $d = -a = 8$.
This coincides with the Hilbert series of $\mathcal{Q}_{4,1}$ so that $\mathcal{R}_{4,1} = \mathcal{Q}_{4,1}$,
and $\mathcal{Q}_{4,1}$ is real radical.

When $n = 2$, the Gr\"{o}bner basis of $\mathcal{R}_{4,2}$ contains
$938$ elements. It took just under $10$ hours to compute on a system with 244GB RAM
and only completed when we used the \texttt{Sort->True} option in the \texttt{GroebnerBasis}
command along with elimination order. The Hilbert series is
\[
    \Hilb_{\R[M_{0,4,2}]}(t) = \frac{1 + 22 t^2 + 225 t^4 + 610 t^6 + 610 t^8 + 225 t^{10} + 22 t^{12} + t^{14}}
    {(1 - t^2)^{14}},
\]
which satisfies Equation \eqref{eq:GradedGorenCond} with $d = -a = 14$.
This coincides with the Hilbert series of $\mathcal{Q}_{4,2}$ so that $\mathcal{R}_{4,2} = \mathcal{Q}_{4,2}$,
and $\mathcal{Q}_{4,2}$ is real radical.

The case $n = 3$ appears to be inaccessible; attempts to compute $\mathcal{R}_{4,3}$
all ran out of memory. Using the option \texttt{Sort->True} and elimination order on
\emph{Mathematica} on a 244GB RAM machine, the computation ran out of memory after about 8 hours.

For $n = 4$, $\mathcal{R}_{4,4}$ appears to be out of reach, but we were able to compute $\mathcal{Q}_{4,4}$.
The Hilbert series of $\mathcal{Q}_{4,4}$ is
\[
\begin{split}
    \Hilb_{\mathcal{Q}_{4,4}}(t)
    =&  \Big(1 + 16 t^2 + 108 t^4 + 395 t^6 + 842 t^8 + 1080 t^{10} + 842 t^{12}
    \\&
                + 395 t^{14} + 108 t^{16} + 16 t^{18} + t^{20}\Big)
        / (1 - t^2)^{20}
\end{split}
\]
which satisfies Equation \eqref{eq:GradedGorenCond} with $d = -a = 20$.


\bibliographystyle{amsplain}
\bibliography{CHS}

\end{document}